\documentclass[a4paper,UKenglish]{lipics-v2019}

\title{The independent set problem is FPT for even-hole-free graphs}

\author{Edin Husi\' c}{Department of Mathematics, LSE, Houghton Street, London, WC2A 2AE, United Kingdom}{e.husic@lse.ac.uk}{}{}

\author{St\' ephan Thomass\' e}{Univ Lyon, CNRS, ENS de Lyon, Universit\' e Claude Bernard Lyon 1, LIP UMR5668, France \and Institut Universitaire de France}{stephan.thomasse@ens-lyon.fr}{}{}

\author{Nicolas Trotignon}{Univ Lyon, ENS de Lyon, Universit\' e Claude Bernard Lyon 1, CNRS, LIP, F-69342, LYON Cedex 07, France}{nicolas.trotignon@ens-lyon.fr}{}{}

\authorrunning{E. Husi\' c, S. Thomass\' e, N. Trotignon}

\Copyright{Edin Husi\' c and St\' ephan Thomass\' e and Nicolas Trotignon}



\funding{The second and third named authors are partially supported by the LABEX MILYON (ANR-10-LABX-0070) of Universit\' e de Lyon, within the program ``Investissements d'Avenir" (ANR-11-IDEX-0007)
operated by the French National Research Agency (ANR)}

\acknowledgements{The majority of paper was prepared while the first named author was a student at ENS de Lyon.}

\nolinenumbers 

\hideLIPIcs  

\EventEditors{Bart M. P. Jansen and Jan Arne Telle}
\EventNoEds{2}
\EventLongTitle{14th International Symposium on Parameterized and Exact Computation (IPEC 2019)}
\EventShortTitle{IPEC 2019}
\EventAcronym{IPEC}
\EventYear{2019}
\EventDate{September 11--13, 2019}
\EventLocation{Munich, Germany}
\EventLogo{}
\SeriesVolume{148}
\ArticleNo{21}

\begin{document}

\maketitle

\begin{abstract}
The class of even-hole-free graphs is very similar to the class of perfect graphs, and was indeed a cornerstone in the tools leading to the proof of the Strong Perfect Graph Theorem. However, the complexity of computing a maximum independent set (MIS) is a long-standing open question in even-hole-free graphs. From the hardness point of view, MIS is W[1]-hard in the class of graphs without induced 4-cycle (when parameterized by the solution size). Halfway of these, we show in this paper that MIS is FPT when parameterized by the solution size in the class of even-hole-free graphs. The main idea is to apply twice the well-known technique of augmenting graphs  to extend some initial independent set. 
\end{abstract}

\section{Introduction}
Given a (finite, simple, undirected) graph $G=(V,E)$ we say that a subset of vertices $I \subseteq V$ is \emph{independent} if every two vertices in $I$ are non-adjacent. 
The \emph{maximum independent set problem} is the problem of finding an independent set of maximum cardinality in a given graph $G$. This problem is NP-hard even for planar graphs of degree at most three~\cite{garey2002computers}, unit disk graphs~\cite{DBLP:journals/dm/ClarkCJ90}, and $C_4$-free graphs~\cite{alekseev1982effect}. ,
To see that the independent set problem is NP-hard in the class of $C_4$-free graphs, one can use the following observation by Poljak~\cite{poljak1974note}. 
Namely, $\alpha(G') = \alpha(G) + 1$ where the graph $G'$ is obtained from $G$ by replacing a single
edge with a $P_4$ (i.e., subdividing it twice). 
By replacing every edge with a $P_4$ we obtain a graph that has girth at least nine, and thus MIS is NP-hard for $C_4$-free graphs.
Similarly, MIS is NP-hard for the class of graphs with girth at least $l$, where $l \in \mathbb{N}$ is fixed.

On the contrary, when the input is restricted to some particular class of graphs the problem can be solved efficiently. Examples of such classes are bipartite graphs~\cite{DBLP:journals/siamcomp/HopcroftK73}, chordal graphs~\cite{gavril1972algorithms} and claw-free graphs~\cite{minty1980maximal,sbihi1980algorithme}. The maximum independent set problem is also polynomially solvable when the input is restricted to the class of perfect graphs using the ellipsoid method~\cite{DBLP:journals/combinatorica/GrotschelLS81}, but it remains an open question to find a combinatorial algorithm\footnote{The term combinatorial algorithm is used for an algorithm that does not rely on the ellipsoid method.} in this case. In fact, we do not even have a combinatorial FPT algorithm for the maximum independent set problem on perfect graphs.

\begin{sloppypar}
Closely related to the class of perfect graphs is the class of even-hole-free graphs. The class of even-hole-free graphs was introduced as a class structurally similar to the class of Berge graphs. We say that a graph is \emph{Berge} if and only if it is odd-hole-free and odd-antihole-free, i.e., $\{C_5, C_7, \overline{C_7}, C_9, \overline{C_9}, \dots\}$-free\footnote{Berge graphs are exactly perfect graphs by the Strong Perfect Graph Theorem.}. 
The similarity  follows from the fact that by forbidding $C_4$, we also forbid all antiholes on at least $6$ vertices. Hence, an even-hole-free graph does not contain an antihole on at least $6$ vertices, i.e., it is $\{C_4, C_6, \overline{C_6}, \overline{C_7}, C_8, \overline{C_8} \dots\}$-free.
It should be noted that techniques obtained in the study of even-hole-free graphs were successfully used in the proof of the Strong Perfect Graph Theorem. A decomposition theorem, an algorithm for the maximum weighted clique problem and several other polynomial algorithms for classical problems in subclasses of even-hole-free graphs can be found in survey~\cite{vuvskovic2010even}.
\end{sloppypar}

We denote by $\alpha (G)$ the maximum cardinality of an independent set in a graph $G$. In this paper we consider a parameterized version of the problem, that is we consider the following decision problem.

\begin{center}
\fbox{\parbox{0.97\linewidth}{\noindent
{\sc Independent Set:}\\[.8ex]
\begin{tabular*}{.90\textwidth}{rl}
{\bf Input:} & A graph $G$.\\
{\bf Parameter:}& $k$. \\
{\bf Output:} & {\sc true} if $\alpha(G) \ge k$ and {\sc false} otherwise.
\end{tabular*}
}}
\end{center}

We say that a problem is \emph{fixed parameter tractable} (FPT) parameterized by the solution size $k$, if there is an algorithm running in time  $O(f(k)n^c)$ for some function $f$ and some constant $c$. 
More generally, a problem is fixed parameter tractable with respect to the parameter $k$ (e.g. solution size, tree-width, ...) if for any
instance of size $n$, it can be solved in time $O(f(k)n^c)$ for some fixed $c$. 
Usually, we consider whether a problem is FPT if the problem is already known to be NP-hard. 
In that case, the function $f$ is not in any way bounded by a polynomial. 
In other words, for fixed parameter tractable problems, the difficulty is not in the input size, but rather in the size of the solution (parameter).
In general, the {\sc Independent Set} problem 
is not fixed-parameter tractable (parameterized by the size of solution) unless W[1]$=$FPT or informally, we believe that there is no FPT algorithm for the problem~\cite{downey2012parameterized}. 
Recently, it has been shown that MIS is W[1]-hard for $C_4$-free graphs~\cite{bonnet_et_al:LIPIcs:2019:10218}. Even stronger, the same paper proves that MIS is W[1]-hard in any family of graphs defined by finitely many forbidden induced holes.

While the exact complexity of the maximum independent set problem is still open for the class of even-hole-free graphs, we present a step forward by showing that there is an FPT algorithm for the problem.

\paragraph*{Main idea} Our algorithm is based on the augmentation technique. More precisely, in order to compute a solution of size $k+1$, we compute disjoint solutions of size $k$. The main property we use is that the union of two independent sets in an even-hole-free graph induces a forest. The key-point of our algorithm is that if $W,X$ are disjoint solutions of size $k$, and $Y$ is some (unknown) solution of size $k+1$, then the two trees induced by $X\cup Y$ and $W\cup Y$ are very constrained. This leads to a reduction to the chordal graph case, where MIS is tractable by dynamic programming.

\paragraph*{Preliminaries} 
We consider finite, simple and undirected graphs. 
For a graph $G=(V,E)$ we write $uv \in E$ for an edge $\{u,v\} \in E(G)$, in this case $u$ and $v$ are \emph{adjacent}.
For a vertex $v \in V(G)$ we denote by $N_G(v) = \{u \in V : uv \in E\}$ \emph{the neighborhood} of $v$ and for $W \subseteq V$, we define $N_G(W) =  \cup_{w \in W} N_G(w) \setminus W$. 
We drop the subscript when it is clear from the context.
Let $S \subseteq V$. 
We say that $S$ is \emph{complete} to $W$ if every vertex in $S$ is adjacent to every vertex in $W$.
The \emph{induced subgraph} $G[W]$  is defined as the graph $H = (W, E \cap \binom{W}{2})$ where $\binom{W}{2}$ is the set of all unordered pairs in $W$. 
For a set $A$ we denote by $A^2$ the set of all ordered pairs with elements in $A$.
The graph $G[V \setminus W]$ is denoted $G\setminus W$ and when $W=\{w\}$ we write $G \setminus w$. 
A subset of vertices is called a \emph{clique} if all the vertices are pairwise adjacent.
A chordless cycle on at least four vertices is called a \emph{hole}. A hole is even (resp. odd) if it contains an even (resp. odd) number of vertices.
A \emph{path} is a graph obtained by deleting one vertex of a chordless cycle. 
A path with endvertices $u, v$ is called a $u,v$-path.
Given a path $Z$ and two of its vertices $v,u$ we denote by $vZu$ the smallest subpath of $Z$ containing $v$ and $u$.
An \emph{in-arborescence} is an orientation of a tree in which every vertex apart one (the \emph{root}) has outdegree one.

\section{Reduction steps and augmenting graphs}
\label{section:MinimalAugmentingGraphs}

Our main goal is to show that the following problem is FPT.

\begin{center}
\fbox{\parbox{0.97\linewidth}{\noindent
{\sc Independent Set in Even-Hole-Free Graphs (ISEHF):}\\[.8ex]
\begin{tabular*}{.90\textwidth}{rl}
{\bf Input:} & An even-hole-free graph $G$.\\
{\bf Parameter:}& $k$. \\
{\bf Output:} & An independent set of size $k$ if $\alpha(G) \ge k$ and {\sc false} otherwise.
\end{tabular*}
}}
\end{center}

We define a simpler version of the ISEHF problem where we know more about the structure 
of $G$. Later, we show that it suffices to find an FPT algorithm for the simpler version.

\begin{center}
\fbox{\parbox{0.97\linewidth}{\noindent
{\sc Transversal Independent Set in Even-Hole-Free Graphs (TISEHF):}\\[.8ex]
\begin{tabular*}{.90\textwidth}{rl}
{\bf Input:} & An even-hole-free graph $G$ and a partition of $V(G)$ into cliques $X_1,\dots ,X_k$.\\
{\bf Parameter:}& $k$. \\
{\bf Output:} & An independent set of size $k$ if $\alpha(G) \ge k$ and {\sc false} otherwise.
\end{tabular*}
}}
\end{center}

Note that in TISEHF, an independent set of size $k$ must intersect every clique on exactly one vertex, i.e., it must traverse all cliques.

\begin{lemma}
The ISEHF problem is FPT if and only if the TISEHF problem is FPT.
\end{lemma}
\begin{proof}
Note that the only if implication is obvious, so we assume that we already have an FPT algorithm $\cal A$ for TISEHF, and provide one for ISEHF.  
We claim that it suffices to exhibit an algorithm $\cal B$ running in time $g(k)n^c$ which takes as input the pair $(G,k)$ and either outputs an independent set of size $k$ or a cover of $V(G)$ by $2^{k-1}-1$ cliques. 
Indeed, one then just has to apply algorithm $\cal A$ to every possible choice of $k$ disjoint cliques induced by the $2^{k-1}-1$ cliques which are output by $\cal B$.
We describe $\cal B$ inductively on $k$: If $k=2$, then $G$ is either a clique, or contains two non-adjacent vertices $x,y$. When $k>2$, we compute two non-adjacent vertices $x,y$ (or return the clique $G$). We now apply $\cal B$ to the graph induced by the set $X$ of non-neighbors of $x$: we either get an independent set of size $k-1$ (in which case we are done by adding $x$) or cover $X$ by $2^{k-2}-1$ cliques. We apply similarly $\cal B$ to the set $Y$ of non-neighbors of $y$. Note that $X\cup Y$ covers all vertices of $G$ except the common neighbors $N$ of $x$ and $y$. Since $G$ is $C_4$-free, $N$ is a clique, and therefore we have constructed a cover of $V(G)$ by $2(2^{k-2}-1)+1$ cliques.
\end{proof}

We turn to our main result. 
In the rest of this section we further reduce the problem to a graph together with two particular trees.
Section~\ref{section:bitrees} defines the notion of bi-trees and shows how two trees interact under certain conditions. 
Then, in Section~\ref{section:end}, we prove that bi-trees arising from even-hole-free graphs satisfy these conditions and conclude the algorithm. 

\begin{theorem}\label{theo:trans}
The TISEHF problem is FPT.
\end{theorem}

\begin{proof}
We assume that we have already shown that there is an algorithm $\cal A$ which solves TISEHF$(G,j)$ in time $O(f(j)n^3)$ for every $j\leq k$. Our goal is to extend this by showing that $f(k+1)$ exists. Our input is a partition of $G$ into cliques $X_1,\dots,X_k,X_{k+1}$ (which we call \emph{parts}) and we aim to either find an independent set intersecting all parts or show that none exists. In what follows, we assume that an independent set $Y=\{y_1,\dots,y_k,y_{k+1}\}$ intersecting all parts exists, and whenever a future argument will end up with a contradiction, this will always be a contradiction to the existence of $Y$, and thus our output will implicitly be {\sc false}. 

The first step is to apply $\cal A$ to $X_1,\dots,X_k$ to compute an independent set $W=\{w_1,\dots ,w_k\}$. If it happens that $W\cap Y\neq \emptyset$, we guess which $w_i$ belongs to $Y$ and run $\cal A$ on the $k$ remaining parts in which we have deleted all neighbors of $w_i$. This costs $k$ calls to TISEHF$(G,k)$ which is in our budget. So we may assume that $W$ is disjoint from $Y$, and even stronger that no vertex of $W$ belongs to an independent set of size $k+1$, since one of the previous $k$ calls would have detected it. Moreover, since there is no even hole, $W\cup Y$ induces a forest $T_1$. Note that no vertex of $W$ is isolated in $T_1$ since the parts are cliques. Note also that $T_1$ cannot have a leaf $w_i$ in $W$, since $w_i$ would belong to an independent set of size $k+1$ by exchanging it with $y_i$. Thus every vertex of $W$ has degree at least two in $T_1$. Since the number of edges of $T_1$ is at most $2k$, we have that every vertex of $W$ has degree 2 and $T_1$ is a tree.

As there is only $h(k)$ possible choices for the structure of $T_1$, we call $h(k)$ branches of computations for each of these choices of $T_1$. This means that in each call, we only keep the vertices of the parts $X_i$ which corresponds to the possible neighborhoods of vertices of $W$. For instance, in the call corresponding to a tree $T_1$ in which $w_1$ is adjacent to $y_1$ and $y_2$, we delete all neighbors of $w_1$ in parts $X_3,\dots ,X_{k+1}$ and delete all non-neighbors of $w_1$ in $X_2$ (no further cleaning is needed in $X_1$ since it is a clique). Therefore, we assume that every vertex of $W$ is complete to exactly two parts (including its own) and non-adjacent to others. 
Moreover, we define a \emph{white tree} on vertex set $\{1,\dots ,k+1\}$ by having an edge between $i$ and $j$ if there exists a vertex $w$ of $W$ which is complete to $X_i$ and $X_j$. We will refer to this vertex $w$ as $w_{i,j}$. In what follows, we do not consider anymore that the vertices of $W$ belong to the parts $X_j$ and rather see them as external vertices of our problem. Thus, since we are free to rename the parts, we can assume that $k+1$ is a leaf of the white tree.

This is the crucial point of the algorithm, we have obtained a more structured input, but unfortunately we could not directly take advantage of it to conclude the main theorem. 
Instead, we apply again algorithm $\cal A$ to $X_1,\dots,X_k$ to compute a second independent set $X=\{x_1,\dots ,x_k\}$ (if such an $X$ does not exist, we thus return {\sc false} as $Y$ cannot exist). 
As done previously, we may assume that $X$ is disjoint from $Y$, the tree $T_2$ spanned by $X\cup Y$ can also be guessed, and the degrees of vertices of $X$ in $T_2$ is two (see Figure~\ref{figure1}, down-left). 
We now interpret $T_2$ in a slightly different way: we root $T_2$ at $y_{k+1}$ and orient all the edges toward the root. By doing so, every edge $\{x_i,y_i\}$ gives the arc $y_ix_i$ while the unique neighbor $y_{r(i)}$ of $x_i$, which is different from $y_i$, gives the arc  $x_iy_{r(i)}$.
We now further clean the parts $X_j$ as follows: for every $x_i$, we delete all neighbors of $x_i$ in $X_j$ for $j\neq i,r(i)$, and we delete all non-neighbors of $x_i$ in $X_{r(i)}$. 
We now have two trees which endow our parts: the white tree and the \emph{red in-arborescence} defined on vertex set $\{1,\dots ,k+1\}$ by the arc set $\{ir(i):i=1,\dots, k\}$. 
Our tool is now ready: the correlation between these two trees will provide an $O(k \cdot n^3)$ time algorithm to compute $Y$, or show that $Y$ does not exist. We now turn to a special section devoted to bi-trees, i.e., trees defined on the same set of vertices under some structural constraints.

\section{Bi-trees}
\label{section:bitrees}

Let $V$ be a set of vertices.  A \emph{bi-tree} is a triple $T = (V, A, E)$ where $E \subseteq {V \choose 2}$ is a set
of edges such that $(V, E)$ is a tree and $A \subseteq V^2$ is a set
of arcs such that $(V, A)$ is an in-arborescence.  For convenience, we view edges
of $(V, E)$ as \emph{white} edges, and arcs of $(V, A)$ as \emph{red}
arcs. 

A \emph{separation} of a bi-tree is a triple $(v, X, Y)$ such that:
\begin{itemize}
\item $V$ is partitioned into nonempty sets $\{v\}$, $X$ and $Y$,
\item no white edge has an end in $X$ and an end in $Y$, and
\item no red arc has an end in $X$ and an end in $Y$.
\end{itemize}

When the sets $X$ and $Y$ are clear from the context, we will simply say that $v$ is a separation.
Note that if $(v, X, Y)$ is a separation of a bi-tree $(V, E, A)$,
then $(X\cup \{v\}, A \cap (X\cup \{v\})^2, E \cap {X\cup \{v\} \choose 2})$ is the bi-tree \emph{induced by $T\setminus Y$}. Observe that if the root is not in $X$, then $T\setminus Y$ is rooted at $v$.

Let $T=(V, A, E)$ be a bi-tree and $a, b, v$ be three distinct
vertices of $V$.  Let $P_{ab}$ be a white path from $a$ to $b$, of
length one or two. Let $P_{av}$ be a directed red path, from $a$ to $v$,
of length at least one.  Let $P_{bv}$ be a directed red path, from $b$
to $v$, of length at least one. We suppose that the three paths are
internally vertex disjoint (meaning that if a vertex is in at least
two of the paths, then it must be $a$, $b$ or $v$). Three such paths
are said to form an \emph{obstruction directed to $v$}.

Let $T=(V, A, E)$ be a bi-tree and $a, b, c, d$ be four distinct
vertices of $V$.  Let $P_{ab}$ be a white path from $a$ to $b$,
$P_{bc}$ be a red path which is directed from $b$ to $c$ or from $c$ to
$b$, $P_{cd}$ be a white path from $c$ to $d$ and $P_{da}$ be a red
path which is directed from $d$ to $a$ or from $a$ to $d$. Suppose that
at least one of $P_{ab}$, $P_{cd}$ has length exactly one and that
the four paths are internally vertex disjoint. Four such paths are
said to form an \emph{alternating obstruction}.

 A \emph{bi-path} is a bi-tree $T = (V, A, E)$ on at least two vertices with an ordering $v_1, \dots, v_n$ of $V$ and an integer $t$ such that:
  \begin{itemize}
  \item $A = \{v_1v_2, \dots, v_{n-1}v_n\}$,
  \item $v_1 v_n \in E$,
  \item $1 \leq t \leq n-1$,
  \item if $t\geq 2$, then $\{v_1v_2, \dots, v_1v_t\}\subseteq E$, and
  \item if $t\leq n-2$, then $\{v_{t+1}v_n, \dots, v_{n-1}v_n\}\subseteq E$.
  \end{itemize}

\begin{lemma}\label{lemma:isolatingPath}
  A bi-tree $T = (V, A, E)$ on at least two vertices, with no
  separation, no directed obstruction and no alternating obstruction is a bi-path.
\end{lemma}
\begin{proof}
  \textit{Case~1}: $(V, A)$ contains some vertex with in-degree at
  least~2.

  We choose such a vertex $v$ as close as possible to the root $r$ of
  $(V, A)$.  Since $(V, A)$ is an in-arborescence, $(V, A)\setminus v$
  has at least $m\geq 2$ in-components $A_1$, \dots , $A_m$ and
  possibly one out-component $B$.
  By the choice of $v$, every vertex
  of $B$ has in-degree exactly 1. Therefore
  $(B \cup \{v\}, A\cap (B \cup \{v\})^2)$ is a directed red path from $v$
  to $r$, that we call $Z$.  We now state and prove two claims. 

  \begin{claim}\label{claim1} 
  For any $1 \leq i < j \leq m$, there is no white edge
  with one end in $A_i$ and one end in $A_j$. 
  \end{claim}
  \begin{claimproof}
  Indeed, such an edge
  would yield an obstruction directed to $v$.
  \end{claimproof}
  \begin{claim}\label{claim2}
  For every $1 \leq i \leq m$, there exists a white edge with
  one end in $A_i$ and one end in $B$ (so, in particular, $B$ exists).
  \end{claim}
  \begin{claimproof}
  For otherwise, Claim~\ref{claim1} implies that
  $(v, A_i, V\setminus (A_i \cup \{v\})$ is a separation.
  \end{claimproof}
  
  Let $P = v, \dots ,z$ be the shortest white path such that $z\in B$ where all internal vertices of $P$ are 
  in $A_1\cup \dots \cup A_m$ ($P$ has possibly length~1). 
  By Claim~\ref{claim1}, $P$ contains vertices from at most one component, say possibly $A_2$, among
  $A_1, \dots, A_m$. By Claim~\ref{claim2}, there exists a vertex $x\in A_1$ with a white neighbor $w$ in $B$.  Let $Q$ be the directed red path from $x$ to $v$.

  If $w$ is an internal vertex of $vZz$ then the edge $xw$, the
  directed path $wZz$, the path $P$, and the directed path $Q$ form an
  alternating obstruction.  If $w$ is a vertex of $zZr$ different from
  $z$, then the edge $xw$, the directed path $zZw$, the path $P$, and
  the directed path $Q$ form an alternating obstruction.  If follows
  that $w=z$.

  If $P$ has length greater than~1, then in particular $z$ has a white
  neighbor $y$ in $A_2$.  Now, the white path $xzy$ and the
  in-components $A_1$ and $A_2$ yield an obstruction directed to $v$.
  So, $P$ has length~1. Consider, by Claim~\ref{claim2}, a vertex $y'$ in $A_2$
  with a neighbor in $B$. The previous argument, with $A_1$ and $A_2$ interchanged,  shows that $y'$ is adjacent to
  $z$ (just as we proved that $x$ is adjacent to $z$).  Again, the
  white path $xzy'$ and the red in-components $A_1$ and $A_2$ yield an
  obstruction directed to $v$.

  \medskip
  \textit{Case~2}: Every vertex in $(V, A)$ has in-degree at most~1.

  Since $(V, A)$ is an in-arborescence, it follows that $(V, A)$ is a
  directed path.  Hence, there exists an ordering $v_1, \dots, v_n$ of
  the vertices of $T$ such that $A = \{v_1v_2, \dots, v_{n-1}v_n\}$.  

  \begin{sloppypar}
  Suppose that there exists a white edge $v_iv_j$ with
  $1<i<j<n$. Then there exists a white edge $v_{i'}v_k$ between
  $\{v_1, \dots, v_{i-1}\}$ and $\{v_{i+1}, \dots, v_{n}\}$ for
  otherwise $(v_i, \{v_1, \dots, v_{i-1}\}, \{v_{i+1}, \dots,
  v_{n}\})$ is a separation. 
  If $k<j$ there is an alternating obstruction, and also if $k>j$.  It
  follows that $k=j$. We proved that there exists a white edge $v_{i'}v_j$,
  with $i'<i$. By a symmetric argument, we can prove that there exists
  $j'>j$ and a white edge $v_iv_{j'}$.  Now, the white edges
  $v_{i'}v_j$,  $v_iv_{j'}$ and the red paths $v_{i'}\dots v_{i}$ and
  $v_{j}\dots v_{j'}$ form an alternating obstruction.
  \end{sloppypar}

  Thus there is no white edge $v_iv_j$ with $1<i<j<n$.
  Hence, every white edge is incident to $v_1$ or to $v_n$.  
  If there exist two white edges $v_1v_j$ and $v_iv_n$ with
  $1<i<j<n$, there is an alternating obstruction, again a
  contradiction. Hence, if we define $t$ as the greatest integer in
  $\{2, \dots, n-1\}$ such that $v_1$ is adjacent to $v_t$ in $(V, E)$
  (with $t=1$ if $v_1$ has no white neighbor among
  $v_2, \dots, v_{n-1}$), we have that $v_n$ has no white neighbor
  among $\{v_2, \dots, v_{t-1}\}$.  Since every vertex has a white
  neighbor, it follows that $v_1$ is white-complete (complete in $(V, E)$) to
  $\{v_2, \dots, v_{t}\}$ (when $t\geq 2$).  For the same reason,
  $v_n$ is white-complete to $\{v_{t+1}, \dots, v_{n-1}\}$ (when
  $t\leq n-2$).

  If $t>1$ and $v_t v_n$
  is a white edge, then $(v_t, \{v_1, \dots,
  v_{t-1}\}, \{v_{t+1}, \dots, v_n\})$ is a separation.  
  So, if $t>1$ then $v_1v_n$ is a white edge, and also if $t=1$. 
 \end{proof}

Given two bi-trees $T_1,T_2$ and a vertex $v$ of $T_1$, we denote by $(T_1,v, T_2)$ the bi-tree obtained by \emph{gluing} $T_2$ at $v$ on $T_1$, i.e., by identifying the root of $T_2$ with $v$. 
A \emph{bi-spider} is a bi-tree which is obtained by iteratively gluing bi-paths at the root vertex (see Figure~\ref{figure1}, right; a bi-spider is induced by the set $\{1,3,4,7,5\}$).
Alternatively, a bi-spider is a bi-tree with no directed obstruction and no alternating obstruction, which is either a bi-path or has only the root as a separation vertex. 

Let $T$ be a bi-tree with no directed obstruction and no alternating obstruction.
Note that the previous lemma asserts that $T$ can be obtained by iteratively gluing bi-paths. 
Indeed, a separation $v$ which is chosen as far as possible from the root must isolate a bi-path. 

Consider a vertex $v$ of a bi-tree $T$. Since $T$ can be obtained by iteratively gluing bi-paths, if $v$ is not a separation then it is a vertex in $T$ which is not used in gluing. 
Thus, the following property holds for $T$: every vertex $v$ which is not the root is either a separation vertex, a leaf of the white tree, or a leaf of the red in-arborescence. We use it to obtain the following result:
 
\begin{corollary}\label{cor:main}
A bi-tree $T = (V, A, E)$ on at least two vertices, with no directed obstruction and no alternating obstruction is either a bi-spider, or admits a separation  $(v, X, Y)$ such that
\begin{enumerate}[(a)]
\item\label{c1} $T\setminus Y$ is a bi-spider,
\item\label{c2} $v$ is either a leaf of the red in-arborescence induced by $T\setminus X$ or a leaf of the white tree induced by $T\setminus X$.
\end{enumerate}
\end{corollary}
\begin{proof}
If $T = (V, A, E)$ is not a bi-spider, it has a separation $(v, X, Y)$ distinct from the root, and we assume that among all choices, $v$ is chosen as far as possible from the root $r$ of the red in-arborescence. 
W.l.o.g., we assume that $Y$ contains $r$. 
Then $T\setminus Y$ is a bi-tree rooted at $v$ which can only admit $v$ as a separation. 
Hence, $T\setminus Y$ is a bi-spider. 
Assume moreover that $Y$ is chosen minimum by inclusion for this property (equivalently, $T\setminus Y$ is a maximum bi-spider rooted at $v$).
We claim that  $v$ is not a separation in bi-tree $T\setminus X$. 
If $v$ is a separation in $T\setminus X$ isolating a bi-path, then we have a contradiction to the minimality of $Y$. 
If $v$ is a separation not isolating a bi-path, then we have a contradiction to the choice of $v$.
Hence, $T\setminus X$ is a bi-tree in which $v$ is not a separation. 
Since $v$ is not the root either, it follows that $v$ is a white leaf or a red leaf in $T\setminus X$. 
\end{proof}

\begin{note}\label{findBiSpider}
A separation isolating a bi-spider with the properties~\eqref{c1} and~\eqref{c2} can be found efficiently. In particular, we find a separation $(v, X, Y)$ isolating a path and then take the maximal (inclusion-wise) set $X$ such that $T\setminus Y$ is still a bi-spider.
\end{note}

\section{The end of the proof}
\label{section:end}

We now resume our proof of Theorem~\ref{theo:trans} as follows. Lemma~\ref{lemma:ehfNoObstructions} shows that the bi-trees arising from even-hole-free graphs do not have the obstructions. Hence, we can use the results from Section~\ref{section:bitrees} where we proved that a bi-tree is either a bi-spider or has a separation isolating a bi-spider. Lemma~\ref{lem:bi-spider} gives an algorithm for the problem when the underlying bi-tree is 
a bi-spider. When the bi-tree is obtained by gluing bi-spiders, Lemma~\ref{lem:bi-cut} proves that combining the partial solutions for each of the bi-spiders produces a valid solution.

Let us recall the hypothesis of Theorem~\ref{theo:trans} (see Figure~\ref{figure1}):

\begin{enumerate}
\item The set of vertices of $G$ is partitioned into $k+1$ cliques $X_1,\dots ,X_{k+1}$ and an additional set $W$ consisting of $k$ vertices $w_{a_1b_1},\dots ,w_{a_kb_k}$.
\item Every $w_{a_ib_i}$ is completely joined to the two parts $X_{a_i}$ and $X_{b_i}$ and has no neighbor in the other parts.
\item The set of pairs $E=\{\{a_i,b_i\}:i=1,\dots ,k\}$, seen as edges on the vertex set $V=\{1,\dots,k+1\}$, forms a white tree in which $k+1$ is a leaf. 
\item Every $X_i$, with $1\leq i\leq k$ contains a particular vertex $x_i$.
\item The set $\{x_1,\dots ,x_{k}\}$ is an independent set.
\item For every vertex $x_i$, there is some $r(i)\neq i$ such that $x_i$ is completely joined to $X_{r(i)}\setminus x_{r(i)}$ (which is just $X_{r(i)}$ when $r(i)=k+1$).
\item The vertex $x_i$ is non-adjacent to every vertex of $X_j$, when $j\neq i$ or $j\neq r(i)$.
\item The set of ordered pairs $A=\{(i,r(i)):i=1,\dots ,k\}$, seen as arcs on the vertex set $V=\{1,\dots ,k+1\}$, forms a red in-arborescence rooted at $k+1$.
\end{enumerate}

We then have a bi-tree $T=(V,E,A)$ on the vertex set  $V=\{1,\dots ,k+1\}$. Furthermore, we want to decide if every part $X_i$, with $1\leq i\leq k+1$ contains a particular vertex $y_i$ distinct from $x_i$ and such that the set of these $y_i$'s forms an independent set.

\begin{figure}[t]
\centering
\includegraphics[width = 0.95\textwidth]{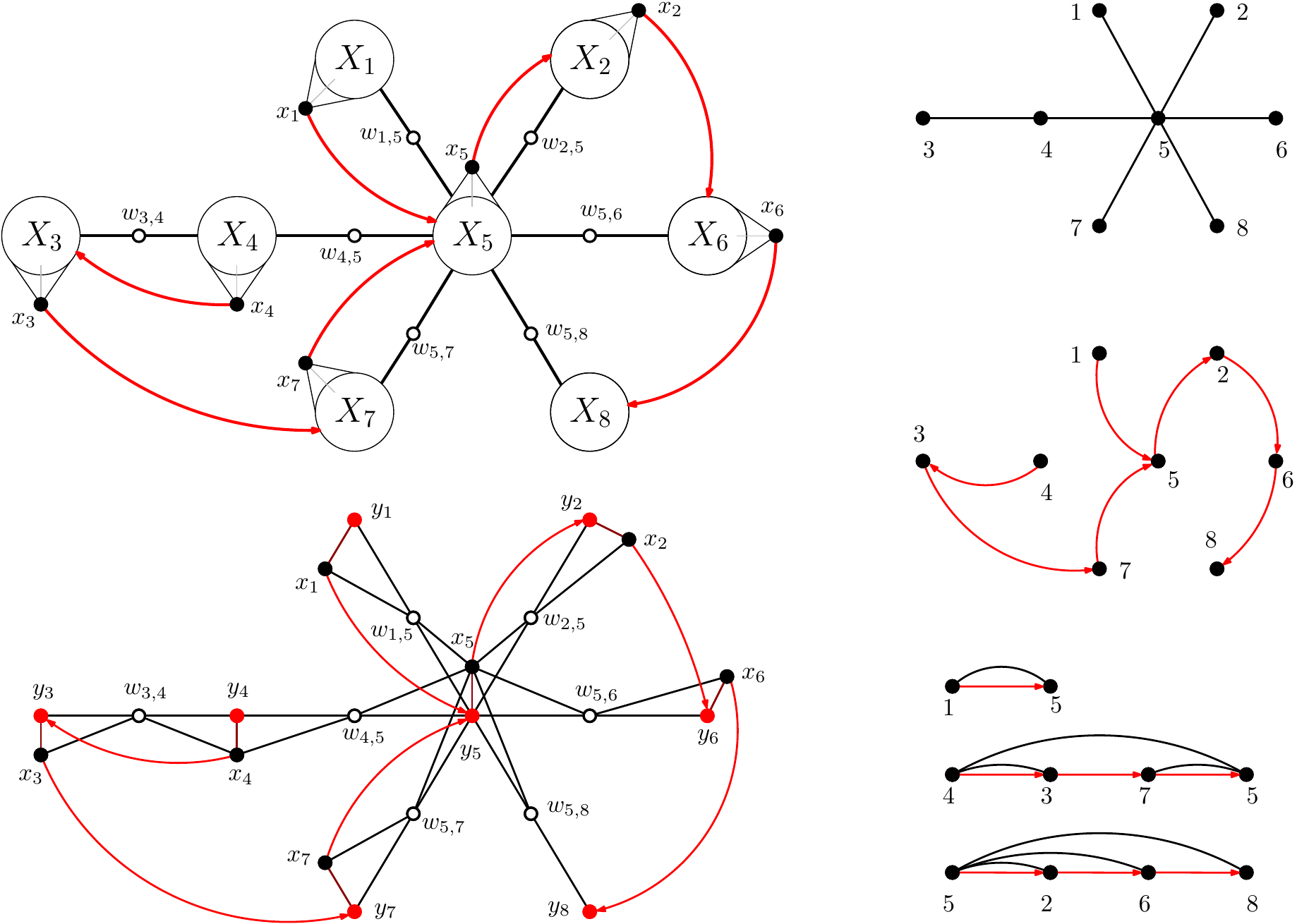}
\caption{Up-left: Graph $G$. Down-left: Set of $y_i$'s. Up-right: White tree. Middle-right: Red in-arborescence. Down-right: Decomposition of bi-tree into bi-paths.}
\label{figure1}
\end{figure}

\begin{lemma}\label{lemma:ehfNoObstructions}
If $G$ has no even holes and a set $Y$ exists, then $T=(V,E,A)$ has no directed obstruction and no alternating obstruction.
\end{lemma}
\begin{proof}
Let us assume that we have a directed obstruction, i.e., we have three distinct vertices $a, b, v$ of $V$, a white path $P_{ab}$ from $a$ to $b$ of
length one or two, a directed red path $P_{av}$ of the form $a=a_0,a_1,\dots ,a_r=v$, and a directed red path $P_{bv}$ of the form $b=b_0,b_1,\dots ,b_s=v$. Our goal is to exhibit an even hole in $G$. The path $P_{ab}$ is either $ab$ or $acb$ and corresponds in $G$ to the path $P_1$ which is either $x_a,w_{ab},x_b$ or $x_a,w_{ac},y_c,w_{cb},x_b$.  The path corresponding to $P_{av}$ is $P_2=x_{a_0},y_{a_1},x_{a_1},\dots ,y_{a_r}$ and the path corresponding to $P_{bv}$ is $P_3=x_{b_0},y_{b_1},x_{b_1},\dots ,y_{b_s}$. Note that $C=P_1\cup P_2\cup P_3$ is an even length cycle. 
Moreover, since each $x_i$ in $C$ is complete to only one class $X_j$ apart from its own, 
there is no chord in $C$, a contradiction.

Let us assume that we have an alternating obstruction on four distinct vertices $a, b, c, d$ of $V$. Two cases arise depending of the direction of the two red paths. When their directions are the same, we have a white path $P_{ab}$ from $a$ to $b$, a red path $P_{bc}$ directed from $b$ to $c$, a white path $P_{cd}$ from $c$ to $d$, and a red path $P_{ad}$ directed from $a$ to $d$. 
By definition of alternating obstruction the four paths are internally vertex disjoint. 
Assuming that $P_{ab}$ is of the form $a=a_0,a_1,\dots ,a_r=b$, we consider in $G$ the corresponding path $P_1=x_{a_0},w_{a_0a_1},y_{a_1},w_{a_1a_2},y_{a_2},w_{a_2a_3},\dots ,x_{a_r}$. 
Assuming that $P_{bc}$ is of the form $b=b_0,b_1,\dots ,b_s=c$, we consider in $G$ the corresponding path $P_2=x_{b_0},y_{b_1},x_{b_1},\dots ,y_{b_s}$. 
Assuming that $P_{ad}$ is of the form $a=d_0,d_1,\dots ,d_u=d$, we consider in $G$ the corresponding path $P_3=x_{d_0},y_{d_1},x_{d_1},\dots ,y_{d_u}$. 
Finally, if $P_{cd}$ is of the form $c=c_0,c_1,\dots ,c_v=d$, we consider in $G$ the corresponding path $P_4=y_{c_0},w_{c_0c_1},y_{c_1},w_{c_1c_2},y_{c_2},\dots ,y_{c_v}$.

When the red paths are in the opposite direction; we have a white path $P_{ab}$ from $a$ to $b$, a red path $P_{bc}$ directed from $b$ to $c$, a white path $P_{cd}$ from $c$ to $d$ and a red path $P_{da}$ directed from $d$ to $a$. 
Again, the four paths are internally vertex disjoint. 
Assuming that $P_{ab}$ is of the form $a=a_0,a_1,\dots ,a_r=b$, we consider in $G$ the corresponding path $P_1=y_{a_0},w_{a_0a_1},y_{a_1},w_{a_1a_2},y_{a_2},w_{a_2a_3},\dots ,x_{a_r}$. 
Assuming that $P_{bc}$ is of the form $b=b_0,b_1,\dots ,b_s=c$, we consider in $G$ the corresponding path $P_2=x_{b_0},y_{b_1},x_{b_1},\dots ,y_{b_s}$. 
Assuming that $P_{da}$ is of the form $d=d_0,d_1,\dots ,d_u=a$, we consider in $G$ the corresponding path $P_3=x_{d_0},y_{d_1},x_{d_1},\dots ,y_{d_u}$. Finally, if $P_{cd}$ is of the form $c=c_0,c_1,\dots ,c_v=d$, we consider in $G$ the corresponding path $P_4=y_{c_0},w_{c_0c_1},y_{c_1},w_{c_1c_2},y_{c_2},\dots ,x_{c_v}$.

Note that both $P_1,P_4$ are even length paths, and $P_2,P_3$ are odd length. Consequently $C=P_1\cup P_2\cup P_3\cup P_4$ is an even length cycle. Moreover, no chord can arise so $C$ is an even hole, a contradiction.
\end{proof}

By Corollary~\ref{cor:main}, the bi-tree $T=(V,E,A)$ is either a bi-spider, or has a separation $i$ isolating a bi-spider. We first conclude in the case of bi-spiders.

\begin{lemma}
\label{lem:bi-spider}
If $T$ is a bi-spider then there is an $O(n^3)$ time algorithm which computes $Y$ or shows that $Y$ does not exist.
\end{lemma}

\begin{proof}
Recall that a bi-spider is a graph obtained by iteratively gluing bi-paths at the root vertex. 
Denote with $T_1,  \dots, T_l$ the bi-paths glued at the root vertex $k+1$ to obtain $T$. 
Moreover, assume that the in-arborescence $T_j$ is a directed path ${j}_1, \dots, {j}_{s_j} = k+1$ for $1\le j \le l$.
Since each $T_j$ is a bi-path, there is a vertex $w_{{j}_1, j_{s_j}}$ and for some value $t_j \in \{2,\dots, s_j\}$ (if any) we have the vertices $\{w_{j_1,j_2},\dots ,w_{j_1,j_{t_j}}\}$ and $\{w_{j_{t_j+1}, j_{s_j} }, \dots ,w_{j_{s_j-1}, j_{s_j}}\}$ (see Figure~\ref{figure:bi-spider}).  

\begin{figure}[t]
\centering
\includegraphics[width = \textwidth]{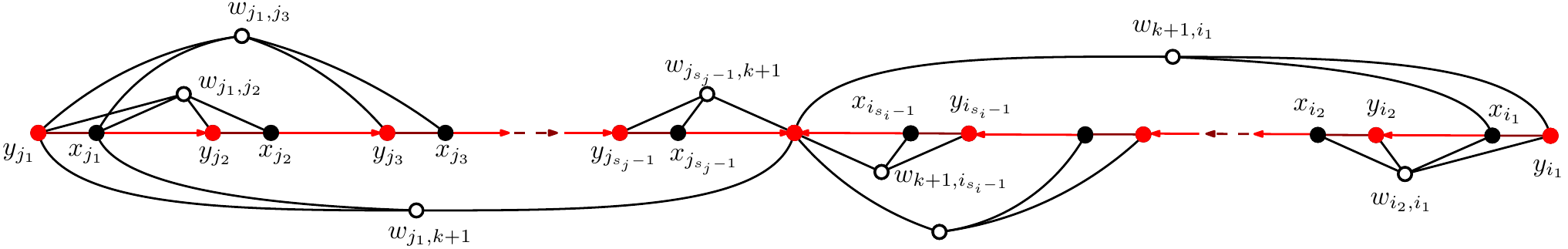}
\caption{An example for Lemma~\ref{lem:bi-spider}.}
\label{figure:bi-spider}
\end{figure}

We decide if $Y$ exists in two phases.
First, for every $1\le j \le l$ we find the set $Y_{j_1}$ of all vertices $y_{j_1}$ which are contained in an independent set of size $t_j$ intersecting $X_{j_1}, \dots, X_{j_{t_j}}$. (Intuitively, $Y_{j_1}$ is the of vertices which can be extended to an independent set traversing $X_{j_1}, \dots, X_{j_{t_j}}$, i.e., all the parts that have a common white neighbor with $y_{j_1}$ except $X_{k+1}$.)
Clearly, if $Y_{j_1}$ is empty for some $j$ then the set $Y$ does not exist. 

Secondly, we find the set $Y_{k+1}$ of vertices $y_{k+1}$ which are contained in an independent set of size $k- \sum_{j=1}^{l} t_j$ intersecting $Y_{j_1}$ and $X_{j_{t_j + 1}}, \dots X_{j_{s_{j}-1}}$ for all $1\le j \le l$. 
(Intuitively, $Y_{k+1}$ is the set of vertices which can be extended to an independent set traversing all the parts that have a common white neighbor with $y_{k+1}$.) Again, if $Y_{k+1}$ is empty then the set $Y$ does not exists. 

We first assume that we have the sets $Y_{j_1}$'s and $Y_{k+1}$ and show how to conclude the lemma in this case. Later, we show that the sets are easy to find. 
Let $y_{k+1} \in Y_{k+1}$ and let $\displaystyle J = \{y_{k+1}\} \cup_{j=1}^l \{y_{j_1}\} \cup_{j=1}^l \{y_{j_{t_j +1}}, \dots,  y_{j_{s_j -1}}\}$ be an independent set of size $k- \sum_{j=1}^{l} t_j$ intersecting all  $Y_{j_1}$ and $X_{j_{t_j + 1}}, \dots X_{j_{s_{j}-1}}$.
For each $1\le j \le l$, denote with $I_j = \{y_{j_1}, \dots, y_{j_{t_j}} \}$ an independent set which contains $y_{j_1}$ and intersects $X_{j_1}, \dots, X_{j_{t_j}}$.
Observe that the set $Y = J \cup_{j=1}^l I_j$ intersects each part of the graph. 
It suffices to prove the following claim. 

\begin{claim}
 $J \cup_{j=1}^l I_j$ is also an independent set.
\end{claim}
\begin{claimproof}
For the sake of contradiction suppose otherwise. We consider two cases. 
Either there is an edge with one end in $J$ and the other end in $I_j$ for some $j$, or there is an edge with ends in $I_j$ and $I_i$ for some $1\le i <  j\le l$. Let us deal with them respectively. 

The mentioned edge is of the form $y_{j_p} y_{i_q}$ where $p \le t_j$ and $t_i < q$ (possibly $j=i$) by definition of $I_j$ and $J$. Choose smallest such $p$.
Observe that $p \neq 1$ since $y_{j_1}$ is a vertex of both $J$ and $I_j$.
If $i_q \neq k+1$ then
$$y_{j_p}, x_{j_{p-1}}, \dots, y_{j_2}, x_{j_1}, w_{j_1, j_{s_j}} (= w_{j_1, k+1}), y_{k+1}, w_{k+1, i_q} (= w_{i_{s_i}, i_q}), y_{i_q}$$ is a cycle of even length. 
Moreover, the cycle is induced by the choice of $p$ and since $\{y_{j_2}, \dots, y_{j_p} \}$ is an independent set, a contradiction. An analogous situation arises if $i_q = k+1$.

Now, we deal with the second case where there is an edge $y_{j_p} y_{i_q}$ where $p \le t_j$, $q \le t_i$  and $j \neq i$. Choose largest such $q$.
It might happen that $p =1$ or $q=1$, but not both since  $y_{j_1}, y_{i_1} \in J$. Without loss of generality, $p \neq 1$. 
Then $$y_{j_p}, x_{j_{p-1}}, \dots, y_{j_2}, x_{j_1}, w_{j_1, j_{s_j}} (= w_{j_1, k+1}), y_{k+1}, x_{i_{s_i -1}}, y_{i_{s_i -1}}, \dots, x_{i_q}, y_{i_q}$$ is an even cycle. 
By the previous case there is no edge between $y_{k+1}$ and $I_j \cup I_i$
Moreover, by the choice of $q$, we deduce that the even cycle is induced, a contradiction.
\end{claimproof}

It remains to show how to find the sets $Y_{j_1}$'s and $Y_{k+1}$.
For the rest of the proof we only use the white tree. 
Observe that it suffices to prove the following (by setting $p = j_1$ for all $j$ and then $p=k+1$). 

\begin{claim}
Let $y_p \in X_p$ and let $G'$ be the graph induced by $X_i$ such that $pi$ is an edge in the bi-tree $T$. Remove neighbors of $y_p$ in $G'$. Then $G'$ is chordal.
\end{claim} 
\begin{claimproof}
For a contradiction, assume that $H$ is an odd hole in $G'$. 
Each part of $G'$ is a clique and, thus, contains at most two vertices of $H$. 
Therefore, there exist an induced path on three vertices $y_a,y_b,y_c$ of $H$, with $y_a, y_b, y_c$ in different parts $X_a, X_b, X_c$. By construction there are vertices $w_{p,a}, w_{p,b}$ and $w_{p,c}$. 
Then $y_{p}, w_{p, a}, y_a, y_b,y_c, w_{p, c}$ induces an even hole in $G$, a contradiction. 
Since $G$ is even-hole-free so is $G'$. Hence $G'$ is hole-free.
\end{claimproof}
Now, for each $j$, we can check if $y_{j_1}$ is in $Y_{j_1}$ by finding a maximum independent set in $G' = G[\cup_{i =2}^{t_j} X_i]\setminus N(y_{j_1})$. The latter can be done in $O(n^2)$ since $G'$ is chordal~\cite{gavril1972algorithms}.  
Then, we can check if $y_{k+1}$ is in $Y_{k+1}$ by finding a maximum independent set in $G' = G[\cup_j \{ Y_{j_1} \cup_{i =t_j + 1}^{s_j-1} X_i \}]\setminus N(y_{k+1})$. This can be done in $O(n^2)$ since $G'$ is chordal.
The overall running time follows since each part is used exactly once in some $G'$.
\end{proof}

In fact, the previous algorithm gives a stronger result:

\begin{corollary}
\label{cor:cut} 
When $T$ is a bi-spider, there is an $O(n^3)$ time algorithm which computes all vertices $y_{k+1}$ which belong to an independent set of size $k+1$.
\end{corollary}

We now deal with the case when $i$ is a separation isolating a bi-spider. 
By Corollary~\ref{cor:main} bi-tree $T$ admits a separation $(i, B, C)$ isolating a bi-spider $T\setminus C$ such that  $i$ is either a red leaf or a white leaf in $T\setminus B$.
Recall that the vertex $k+1$ is a leaf of the white tree, hence, as a separation, $i$ is not equal to $k+1$.
In particular, the vertex $x_i$ exists. 
Moreover, since $T\setminus C$ is a bi-spider it follows that $k+1 \in C$. 
As before, assuming the set $Y$ exists, we obtain the following lemma.

\begin{lemma}\label{lem:bi-cut}
There is no edge from some $y_j$ with $j\in B\setminus i$ to some vertex $u \in X_s$ with $s\in C$. 
\end{lemma}
\begin{proof}
We denote by $r$ the root of $T$. As argued above $r \in C$ ($r= k+1$).
For the sake of contradiction suppose that there is an an edge $y_ju$.

Let us consider bi-spider $T\setminus C$.
There is a red path $j=j_0,\dots ,j_a=i$ in $(V, A)$ which can be turned into an induced path $P_0=y_{j_0},x_{j_0},y_{j_1},x_{j_1},\dots ,y_{j_a},x_{j_a}$ in $G$ from $y_j$ to $x_i$ with odd length. 
There is also a white path $j=b_0,\dots ,b_d=i$ in $(V, E)$ which can be turned into an induced path $P_1=y_{b_0},w_{b_0b_1},y_{b_1},w_{b_1b_2},\dots ,x_{b_d}$ in $G$ from $y_j$ to $x_i$ with even length. 
Now, in order to conclude the lemma it suffices to find a $u, x_i$ path $P$ such that $P.P_0$ and $P.P_1$ induce cycles. Then, since $P_0$ and $P_1$ are of different parity a contradiction arises. 
In the rest of the proof we show how to find $P$.
 
First, observe that since $T\setminus C$ contains a white subtree, $u$ is non-adjacent to $y_i$ or to any $y_q$ where $q\in B$ and $q\neq j$ since it would yield an even hole (there is an even path between any two different vertices $y_p, y_q$). Hence, $u$ is adjacent to $y_j$ and non-adjacent to all other vertices in $P_0$ and $P_1$.

By Corollary~\ref{cor:main}, $i$ is either a red leaf or a white leaf in $T\setminus B$. We consider two cases.

\smallskip
\textit{Case 1:} $i$ is a red leaf. Then there is a (an undirected) red path $i = i_0, \dots, i_s = s$ in $T\setminus B$, which can be turned into an induced path $P = x_{i_0}, y_{i_1}, x_{i_1}, \dots, x_{i_{s-1}, u}$ in $G$ from $x_i$ to $u$. By construction, this path is induced. Moreover, since $i$ is a red leaf in $T\setminus B$ it follows that $y_{i_1} \neq y_i$. Therefore, both $P.P_0$ and $P.P_1$ induce cycles, i.e., there is no chord with one end in $P$ and the other in $P_0$ or $P_1$.

\smallskip
Note that the same argument holds whenever the red path $i = i_0, \dots, i_s = s$ does not contain $y_i$.
Hence, the only remaining case is when $i$ is on the red directed path from $s$ to $r$ in $T\setminus B$.
Denote with $Q$ the directed red $sr$ path in  $T\setminus B$. 

\smallskip
\textit{Case 2:} $i$ is a white leaf and $i \in Q$. 
Let $Q' = iQr$ be subpath of $Q$ starting at $i$ and ending at $r$. 
Since $i$ is not a separation of $T\setminus B$, there exists a white path $s = s_0, \dots, s_t$ 
connecting $s$ and $Q'$. Moreover, the path does not contain $i$.  
We choose the shortest such $s,Q'$ path. 
This path can be turned into an induced path $P_2 = u, w_{s_0, s_1}, y_{s_1}, \dots, w_{s_{t-1}, s_t}, y_{s_t}$ in $G$ with endpoints $u$ and $y_{s_t}$.
By the above $y_i \not \in P_2$ and also no vertex $w_{i, . }$ is used in $P_2$.

Consider the directed path $iQ s_t$ $(= i Q' s_t)$. Denote it as $i=i_0, i_1, \dots, i_{l} = s_t$. It can be turned into an induced path $P_3 = x_{i_0}, y_{i_1}, \dots y_{i_l}$ in $G$ with endpoints $x_i$ and $y_{s_t}$. 
Then $P_2. P_3$ is a $u, x_i$ path in $G$. The concatenation $P_2. P_3$ might not be an induced path, but we can shorten it to obtain an induced $u x_i$ path $P$ in $G$. Now, it can be checked that $P.P_0$ and $P. P_1$ induce cycles since $P$ does not use $y_i$ or any of the vertices $w_{i, .}$. 
\end{proof}

We are now ready to show that there is an $O(k \cdot n^3)$ time algorithm which computes $Y$ when  $T=(V,E,A)$ is a bi-tree. If $T$ is a bi-spider, we are done by Lemma~\ref{lem:bi-spider}. Otherwise, by Corollary~\ref{cor:main}, there is a separation $(i, B, C)$ which isolates a bi-spider $T\setminus C$.
By Lemma~\ref{lem:bi-cut}, one can delete all vertices $y_j \in X_j$ for $j\in B\setminus i$ with a neighbor $u \in X_k$ for $k\notin B$, and this reduction is sound since no candidate $y_j$ can have such an edge. Now, by Corollary~\ref{cor:cut}, one can compute in $O(n^3)$ time the set $X'_i \subseteq X_i$ of vertices, each of which extends, in the bi-spider $T\setminus C$, to an independent set of size $|B|$. 
From the bi-spider $T\setminus C$, we only keep these vertices $X'_i$. 
Observe that the number of parts has now decreased by at least one. 
We repeat this process until we either construct $X'_{k+1}$ or conclude that this set is empty. 
If $X'_{k+1} \neq \emptyset$, then we can reconstruct the set $Y$.
The total time is $O(k \cdot n^3)$.
\end{proof}



\bibliography{finalFPT}

\begin{thebibliography}{10}

\bibitem{alekseev1982effect}
VE~Alekseev.
\newblock The effect of local constraints on the complexity of determination of
  the graph independence number.
\newblock {\em Combinatorial-algebraic methods in applied mathematics}, pages
  3--13, 1982.

\bibitem{bonnet_et_al:LIPIcs:2019:10218}
{\'E}douard Bonnet, Nicolas Bousquet, Pierre Charbit, St{\'e}phan Thomass{\'e},
  and R{\'e}mi Watrigant.
\newblock {Parameterized Complexity of Independent Set in H-Free Graphs}.
\newblock In Christophe Paul and Michal Pilipczuk, editors, {\em 13th
  International Symposium on Parameterized and Exact Computation (IPEC 2018)},
  volume 115 of {\em Leibniz International Proceedings in Informatics
  (LIPIcs)}, pages 17:1--17:13, Dagstuhl, Germany, 2019. Schloss
  Dagstuhl--Leibniz-Zentrum fuer Informatik.

\bibitem{DBLP:journals/dm/ClarkCJ90}
Brent~N. Clark, Charles~J. Colbourn, and David~S. Johnson.
\newblock Unit disk graphs.
\newblock {\em Discrete Mathematics}, 86(1-3):165--177, 1990.

\bibitem{downey2012parameterized}
Rodney~G. Downey and Michael~Ralph Fellows.
\newblock {\em Parameterized complexity}.
\newblock Springer Science \& Business Media, 2012.

\bibitem{garey2002computers}
Michael~R Garey and David~S Johnson.
\newblock {\em Computers and intractability}, volume~29.
\newblock wh freeman New York, 2002.

\bibitem{gavril1972algorithms}
F{\u{a}}nic{\u{a}} Gavril.
\newblock Algorithms for minimum coloring, maximum clique, minimum covering by
  cliques, and maximum independent set of a chordal graph.
\newblock {\em SIAM Journal on Computing}, 1(2):180--187, 1972.

\bibitem{DBLP:journals/combinatorica/GrotschelLS81}
Martin Gr{\"{o}}tschel, L{\'{a}}szl{\'{o}} Lov{\'{a}}sz, and Alexander
  Schrijver.
\newblock The ellipsoid method and its consequences in combinatorial
  optimization.
\newblock {\em Combinatorica}, 1(2):169--197, 1981.

\bibitem{DBLP:journals/siamcomp/HopcroftK73}
John~E. Hopcroft and Richard~M. Karp.
\newblock An n\({}^{\mbox{5/2}}\) algorithm for maximum matchings in bipartite
  graphs.
\newblock {\em {SIAM} J. Comput.}, 2(4):225--231, 1973.

\bibitem{minty1980maximal}
George~J Minty.
\newblock On maximal independent sets of vertices in claw-free graphs.
\newblock {\em Journal of Combinatorial Theory, Series B}, 28(3):284--304,
  1980.

\bibitem{poljak1974note}
Svatopluk Poljak.
\newblock A note on stable sets and colorings of graphs.
\newblock {\em Commentationes Mathematicae Universitatis Carolinae},
  15(2):307--309, 1974.

\bibitem{sbihi1980algorithme}
Najiba Sbihi.
\newblock Algorithme de recherche d'un stable de cardinalit{\'e} maximum dans
  un graphe sans {\'e}toile.
\newblock {\em Discrete Mathematics}, 29(1):53--76, 1980.

\bibitem{vuvskovic2010even}
Kristina Vu{\v{s}}kovi{\'c}.
\newblock Even-hole-free graphs: a survey.
\newblock {\em Applicable Analysis and Discrete Mathematics}, pages 219--240,
  2010.

\end{thebibliography}

\end{document}